\documentclass[11pt, a4paper, reqno]{amsart}

\usepackage{amsmath, amssymb, amsthm}
\usepackage[T1]{fontenc}
\usepackage[utf8]{inputenc}
\usepackage[all]{xy}
\usepackage[
left=2.5cm,right=2.5cm,top=2.5cm,bottom=2.5cm,a4paper
  ]{geometry}
\usepackage{hyperref}
\usepackage{indentfirst}
\usepackage[last]{changes}
\definechangesauthor[color=blue]{KY}

\def\A{{\mathbb A}}

\def\C{{\mathbb C}}

\def\F{{\mathbb F}}

\def\Q{{\mathbb Q}}
\def\R{{\mathbb R}}

\def\Z{{\mathbb Z}}

\theoremstyle{plain}
\newtheorem{theorem}{Theorem}[section]
\newtheorem{lemma}[theorem]{Lemma}
\newtheorem{proposition}[theorem]{Proposition}
\newtheorem{corollary}[theorem]{Corollary}
\newtheorem{remark}[theorem]{Remark}
\newtheorem{definition and lemma}[theorem]{Definition and Lemma}
\newtheorem{problem}[theorem]{Problem}

\theoremstyle{definition} 
\newtheorem{definition}[theorem]{Definition}

\newtheorem{example}[theorem]{Example}

\theoremstyle{remark} 

\DeclareFontFamily{U}{wncy}{}
\DeclareFontShape{U}{wncy}{m}{n}{<->wncyr10}{}
\DeclareSymbolFont{mcy}{U}{wncy}{m}{n}
\DeclareMathSymbol{\Sha}{\mathord}{mcy}{"58}

\title[Automorphism groups of polarized abelian fourfolds]{On the automorphism groups of simple polarized abelian fourfolds over finite fields and Jordan constants}
\author{WonTae Hwang}
\address{School of Mathematics, Korea Institute for Advanced Study, 85 Hoegiro, Dongdaemun-gu, Seoul 02455, South Korea}
\email{hwangwon@kias.re.kr}

\begin{document}



\maketitle

\begin{abstract}
We give a classification of maximal elements of the set of finite groups that can be realized as the full automorphism groups of simple polarized abelian fourfolds over finite fields. As an application, we compute the Jordan constants of the automorphism groups of simple abelian fourfolds over finite fields.  
\end{abstract}

\section{Introduction}
Let $k$ be a finite field, and let $X$ be an abelian variety of dimension $g$ over $k.$ We denote the endomorphism ring of $X$ over $k$ by $\textrm{End}_k(X)$. It is a free $\Z$-module of rank $\leq 4g^2.$ We also let $\textrm{End}^0_k(X)=\textrm{End}_k(X) \otimes_{\Z} \Q.$ This $\Q$-algebra $\textrm{End}_k^0(X)$ is called the \emph{endomorphism algebra of $X$ over $k.$} Then $\textrm{End}_k^0(X)$ is a finite dimensional semisimple algebra over $\Q$ with $\textrm{dim}_{\Q} \textrm{End}_k^0(X) \leq 4g^2.$ Moreover, if $X$ is simple, then $\textrm{End}_k^0(X)$ is a division algebra over $\Q$. Now, it is well known that $\textrm{End}_k (X)$ is a $\Z$-order in $\textrm{End}_k^0(X).$ The group $\textrm{Aut}_k(X)$ of the automorphisms of $X$ over $k$ is not finite, in general. But once we fix a polarization $\mathcal{L}$ on $X$, then the group $\textrm{Aut}_k(X,\mathcal{L})$ of the automorphisms of the polarized abelian variety $(X,\mathcal{L})$ is finite. This fact leads us to the following natural problem.
\begin{problem}\label{main prob}
Given a finite group $G$ and an integer $g \geq 1,$ do there exist a field $k$ and a polarized abelian variety $(X,\mathcal{L})$ of dimension $g $ over $k$ such that $G=\textrm{Aut}_k (X,\mathcal{L})?$ 
\end{problem}
\indent Regarding the above problem, in characteristic zero, Birkenhake and Lange \cite[\S13]{BL(2004)} (resp.\ Birkenhake, Gonz$\acute{\textrm{a}}$lez, and Lange \cite{BGL(1999)}) classified all finite subgroups $G$ of the automorphism group $\textrm{Aut}_k(X)$ that is maximal in the isogeny class of $X$, where $X$ is an arbitrary complex abelian surface (resp.\ a complex torus of dimension $3$).\\ 
\indent The goal of this paper is to classify all finite groups that can be realized as the automorphism group of a simple polarized abelian variety $(X,\mathcal{L})$ of dimension $4$ over a finite field $k,$ which are maximal in the sense of Definition \ref{def 20} below, and hence, to give an answer for Problem \ref{main prob} for the case when $g=4,$ $X$ is simple, and $k$ is a finite field. \\
\indent Along this line, the author \cite{8} (resp.\ \cite{9}) gave such a classification for arbitrary polarized abelian surfaces (resp.\ simple polarized abelian varieties of odd prime dimension) over finite fields. One of the main results of \cite{8} implies that the automorphism groups of simple polarized abelian surfaces over finite fields need not be abelian (see \cite[Theorem 6.5]{8}), while the main theorem of \cite{9} says that the automorphism groups of simple polarized abelian varieties of odd prime dimension over finite fields are necessarily cyclic (see \cite[Theorem 4.1]{9}). The methods that we used to prove the main results of \cite{8} and \cite{9} were essentially somewhat different. In this paper, as next step toward goal of considering other higher dimensional cases, we will give a classification of the automorphism groups of simple polarized abelian varieties of dimension $4$ over finite fields by exhibiting an explicit list of such groups. (This kind of classification might not only be interesting in its own sake, but also have some applications to other areas of mathematics such as Group Theory (see Section \ref{Jordan sec} below for an example) and Algebraic Geometry (for instance, we can try to see the geometry of an algebraic variety obtained from an abelian variety $X$ quotiented by a finite group $G$ of automorphisms of $X$.) One main difficulty arises from the fact that we need to deal with central simple algebras with larger even $\Q$-dimension compared to those of \cite{8} and \cite{9}, which, in turn, results in the occurrence of some new finite groups that we need to consider. Consequently, we need a somewhat different method to prove the realizability of those newly introduced groups as the automorphism group of some simple polarized abelian fourfolds over finite fields.  \\

Our first main result of the paper is the following
\begin{theorem}\label{main theorem2}
  The possibilities for the (maximal) automorphism group of a simple polarized abelian variety $(X, \mathcal{L})$ of dimension $4$ over a finite field is given as follows: \\
(1) (Cyclic groups) $C_n$ for $n \in \{2,4,6,8,10,12,16,20,24,30 \}$; \\
(2) (Non-cyclic groups) $C_5 \rtimes C_8$, $C_3 \rtimes C_{16},$ $C_5 \rtimes C_{16}$.
\end{theorem}
One notable thing about the result is that many of the groups which appeared in \cite[Theorem 6.5]{8} are not realizable in our current case because of some natural number theoretic reasons. For more details, see Theorem \ref{main thm} below.  \\

As an application of Theorem \ref{main theorem2}, we also compute the Jordan constants of the automorphism groups of simple abelian fourfolds over finite fields. In this regard, another main result of this paper is the following
\begin{theorem}
If $G = \textrm{Aut}_k (X)$ for some simple abelian fourfold $X$ over a finite field $k,$ then $G$ is a Jordan group and its Jordan constant $J_G$ is equal to either $1$ or $2$ or $4.$ 
\end{theorem}
For the definition of Jordan groups and Jordan constants, see Definition \ref{Jordan} below.

This paper is organized as follows: In Section~\ref{prelim}, we introduce several facts which are related to our desired classification. More explicitly, we will recall some facts about endomorphism algebras of simple abelian varieties ($\S$\ref{end alg av}), the theorem~of Tate ($\S$\ref{thm Tate sec}), and Honda-Tate theory ($\S$\ref{thm Honda}). In Section~\ref{findiv}, we find all the finite groups that can be embedded in certain division algebras by a detailed analysis using a paper of Amitsur \cite{1}. In Section~\ref{main}, we obtain the desired classification using the facts that were introduced in the previous sections, together with several auxiliary results that are proved in this section. Finally, in Section~\ref{Jordan sec}, we use our main result to compute the Jordan constants of the automorphism groups of simple abelian fourfolds over finite fields, as an application. \\
\indent In the sequel, let $g \geq 1$ be an integer and let $q=p^a$ for some prime number $p$ and an integer $a \geq 1,$ unless otherwise stated. Also, let $\overline{k}$ denote an algebraic closure of a field $k$. Finally, for an integer $n \geq 1,$ we denote the cyclic group of order $n$ (resp.\ the dihedral group of order $2n$) by $C_n$ (resp.\ $D_n$). 
\section{Preliminaries}\label{prelim}
In this section, we recall some of the facts in the general theory of abelian varieties over a field. Our main references are \cite{3} and \cite{12}.
\subsection{Endomorphism algebras of simple abelian varieties of dimension $4$ over finite fields}\label{end alg av}
In this section, we classify all the possible endomorphism algebras of simple abelian varieties of dimension $4$ over finite fields. To this aim, we first review the general theory on the subject: let $X$ be a simple abelian variety of dimension $g$ over a finite field $k.$ Then it is well known that $\textrm{End}_k^0(X)$ is a division algebra over $\Q$ with~$2g \leq \textrm{dim}_{\Q} \textrm{End}_k^0(X) < (2g)^2.$ Before giving our first result, we also recall Albert's classification. We choose a polarization $\lambda \colon X \rightarrow \widehat{X}$ where $\widehat{X}$ denotes the dual abelian variety of $X$. Using the polarization $\lambda,$ we can define an involution, called the \emph{Rosati involution}, $^{\vee}$ on the algebra $\textrm{End}_k^0(X).$ (For a more detailed discussion about the Rosati involution, see \cite[\S20]{12}.) In this way, to a polarized abelian variety $(X,\lambda)$ we associate the pair $(D, ^{\vee})$ with $D=\textrm{End}_k^0(X)$ and $^{\vee}$, the Rosati involution on $D$. 
Let $K$ be the center of $D$ so that $D$ is a central simple $K$-algebra, and let $K_0 = \{ x \in K~|~x^{\vee} = x \}$ be the subfield of symmetric elements in $K.$ By a theorem of Albert (see \cite[Application I in \S21]{12}), $D$ (together with $^\vee$) is of one of the following four types: \\
\indent (i) Type I: $K_0 = K=D$ is a totally real field. \\
\indent (ii) Type II: $K_0 = K$ is a totally real field, and $D$ is a quaternion algebra over $K$ with $D\otimes_{K,\sigma} \R \cong M_2(\R)$ for every embedding $\sigma \colon K \hookrightarrow \R.$  \\
\indent (iii) Type III: $K_0 = K$ is a totally real field, and $D$ is a quaternion algebra over $K$ with $D\otimes_{K,\sigma} \R \cong \mathbb{H}$ for every embedding $\sigma \colon K \hookrightarrow \R$ (where $\mathbb{H}$ is the Hamiltonian quaternion algebra over $\R$). \\
\indent (iv) Type IV: $K_0 $ is a totally real field, $K$ is a totally imaginary quadratic field extension of $K_0$, and $D$ is a central simple algebra over $K$. \\
\indent Keeping the notations as above, we let
\begin{equation*}
  e_0= [K_0 :\Q],~~~e=[K:\Q],~~~\textrm{and}~~~d=[D:K]^{\frac{1}{2}}.
\end{equation*}

As our last preliminary fact of this section, we note some numerical~restrictions on those values $e_0, e,$ and $d$ in Table \ref{tab1} above, following \cite[\S21]{12}.
\begin{table}
\caption{Numerical restrictions on endomorphism algebras}
\label{tab1}       
\begin{center}
\begin{tabular}{ccc}
\hline\noalign{\smallskip}
 & $\textrm{char}(k)=0$ & $\textrm{char}(k)=p>0$  \\
\noalign{\smallskip}\hline\noalign{\smallskip}
$\textrm{Type I}$ & $e|g$ & $e|g$ \\
$\textrm{Type II}$ & $2e|g$ & $2e|g$ \\
$\textrm{Type III}$ & $2e|g$ & $e|g$ \\
$\textrm{Type IV}$ & $e_0 d^2 |g$ &$ e_0 d |g $ \\
\noalign{\smallskip}\hline
\end{tabular}
\end{center}
\end{table}


Now, we are ready to introduce the following result for the case when $g=4$.
\begin{lemma}\label{poss end alg}
Let $X$ be a simple abelian variety of dimension $4$ over a finite field $k=\F_q$, and let $\lambda \colon X \rightarrow \widehat{X}$ be a polarization. Then $D:=\textrm{End}_k^0(X)$ (together with the Rosati involution $^\vee$ corresponding to $\lambda$) is of one of the following three types: \\
(1) $D$ is a central simple division algebra of degree $4$ over an imaginary quadratic field; \\
(2) $D$ is a central simple division algebra of degree $2$ over a CM-field of degree $4$; \\
(3) $D$ is a CM-field of degree $8$.
\end{lemma}
\begin{proof}
We recall that $X$ is of CM-type (see Corollary \ref{cor TateEnd0}-(c) below), and hence, either $D$ is of Type III or Type IV in Albert's classification. \\
\indent (i) First, we show that $D$ is not of Type III. To this aim, suppose that $D$ is of Type III. By Table \ref{tab1}, we know that $e|g.$ If $e=1,$ then $\textrm{dim}_{\Q} D =4,$ which contradicts the fact that $8 \leq \textrm{dim}_{\Q} D.$ If $e=2$ (resp.\ $e=4$), then $D$ is a quaternion algebra over a totally real field $K=\Q(\pi_X)$ of degree $2$ (resp.\ of degree $4$) where $\pi_X$ denotes the Frobenius endomorphism of $X.$ Since $\pi_X$ is a $q$-Weil number, we have $|\pi_X |=\sqrt{q}$ so that $[K : \Q] \leq 2,$ and hence, we only need to consider the case when $e=2.$ In this case, it follows from Corollary \ref{cor TateEnd0}-(b) that $\dim X =2,$ which is a contradiction. Hence, we can conclude that $D$ cannot be of Type III. \\
\indent (ii) As a result of (i), we may assume that $D$ is of Type IV. By Table \ref{tab1}, the possible pairs $(e_0, d)$ are contained in the following set:
\begin{equation*}
\{(1,1), (1,2), (2,1), (1,4), (2,2), (4, 1)\}
\end{equation*}  
From the above list, we exclude the pairs $(1,1)$, $(2,1),$ and $(1,2),$ as follows: if $e_0=d=1$ (resp.\ $e_0=2, d=1$), then $D$ is an imaginary quadratic field (resp.\ a CM-field of degree $4$), and hence, we have $\dim_{\Q}D \leq 4$ in both cases, which contradicts the fact that $8 \leq \dim_{\Q}D.$ If $e_0=1, d=2,$ then we have $\dim X=2$ by Corollary \ref{cor TateEnd0}-(b), which is also a contradiction.
Now, the desired result follows from Albert's classification.  \\
\indent This completes the proof.
\end{proof}

\begin{remark}\label{odd prime holds}
There is a similar result for the case when $X$ is a simple abelian variety of dimension $g$ for a prime $g \geq 3$ over a finite field $k$ (see \cite[Lemma 2.1]{9}). 
\end{remark}


\subsection{The theorem of Tate}\label{thm Tate sec}
In this section, we recall an important theorem of Tate, and give some of its interesting consequences. \\
\indent Let $k$ be a field and let $l$ be a prime with $l \ne \textrm{char}(k)$. If $X$ is an abelian variety of dimension $g$ over $k,$ then we can introduce the Tate $l$-module $T_l X$ and the corresponding $\Q_l$-vector space $V_l X :=T_l X \otimes_{\Z_l} \Q_l.$ It is well known that $T_l X$ is a free $\Z_l$-module of rank $2g$ and $V_l X$ is a $2g$-dimensional $\Q_l$-vector space. In \cite{16}, Tate showed the following important result.

\begin{theorem}\label{thm Tate}
Let $k$ be a finite field and let $\Gamma = \textrm{Gal}(\overline{k}/k).$ If $l$ is a prime with $l \ne \textrm{char}(k),$ then we have: \\
(a) For any abelian variety $X$ over $k,$ the representation
\begin{equation*}
 \rho_l =\rho_{l,X} \colon \Gamma \rightarrow \textrm{GL}(V_l X)
\end{equation*}
is semisimple. \\
(b) For any two abelian varieties $X$ and $Y$ over $k,$ the map
\begin{equation*}
 \Z_l \otimes_{\Z} \textrm{Hom}_k(X,Y) \rightarrow \textrm{Hom}_{\Gamma}(T_l X, T_l Y)
\end{equation*}
is an isomorphism.
\end{theorem}

Now, we recall that an abelian variety $X$ over a (finite) field $k$ is called \emph{elementary} if $X$ is $k$-isogenous to a power of a simple abelian variety over $k.$ Then, as an interesting consequence of Theorem \ref{thm Tate}, we have the following
\begin{corollary}[{\cite[Theorem 2]{16}}]\label{cor TateEnd0}
  Let $X$ be an abelian variety of dimension $g$ over a finite field $k.$ Then we have:\\
  (a) The center of the endomorphism algebra $\textrm{End}_k^0(X)$ is the subalgebra $\Q[\pi_X]$ where $\pi_X$ denotes the Frobenius endomorphism of $X.$ In particular, $X$ is elementary if and only if $\Q[\pi_X]=\Q(\pi_X)$ is a field. \\
 (b) Suppose that $X$ is elementary. Let $h=f_{\Q}^{\pi_X}$ be the minimal polynomial of $\pi_X$ over $\Q$, and let $f_X$ denote the characteristic polynomial of $\pi_X$. Further, let $d=[\textrm{End}_k^0(X):\Q(\pi_X)]^{\frac{1}{2}}$ and $e=[\Q(\pi_X):\Q].$ Then we have $de =2g$ and $f_X = h^d.$ \\
(c) We have $2g \leq \textrm{dim}_{\Q} \textrm{End}^0_k (X) \leq (2g)^2$ and $X$ is of CM-type. \\
\end{corollary}

For a precise description on the structure of the endomorphism algebra of a simple abelian variety $X$, viewed as a simple algebra over its center $\Q[\pi_X]$, we record the following result.
\begin{proposition}[{\cite[Corollary 16.30 and Corollary 16.32]{3}}]\label{local inv}
  Let $X$ be a simple abelian variety over a finite field $k=\F_q$ and let $K=\Q[\pi_X]$. Then we have: \\
(a) If $\nu$ is a place of $K$, then the local invariant of $\textrm{End}_k^0(X)$ in the Brauer group $\textrm{Br}(K_{\nu})$ is given by
  \begin{equation*}
    \textrm{inv}_{\nu}(\textrm{End}_k^0(X))=\begin{cases} 0 & \mbox{if $\nu$ is a finite place not above $p$}; \\ \frac{\textrm{ord}_{\nu}(\pi_X)}{\textrm{ord}_{\nu}(q)} \cdot [K_{\nu}:\Q_p] & \mbox{if $\nu$ is a place above $p$}; \\ \frac{1}{2} & \mbox{if $\nu$ is a real place of $K$}; \\ 0 & \mbox{if $\nu$ is a complex place of $K$}. \end{cases}
  \end{equation*}
(b) If $d$ is the degree of the division algebra $D:=\textrm{End}_k^0(X)$ over its center $K$ (so that $d=[D:K]^{\frac{1}{2}}$ and $f_X = (f_{\Q}^{\pi_X})^d$), then $d$ is the least common denominator of the local invariants $\textrm{inv}_{\nu}(D).$
\end{proposition}

\subsection{Abelian varieties up to isogeny and Weil numbers: Honda-Tate theory}\label{thm Honda}
In this section, we recall an important theorem of Honda and Tate. To achieve our goal, we first give the following
\begin{definition}\label{qWeil Def}
  (a) A \emph{$q$-Weil number} is an algebraic integer $\pi$ such that $| \iota(\pi) | = \sqrt{q}$ for all embeddings $\iota \colon \Q[\pi] \hookrightarrow \C.$ \\
  (b) Two $q$-Weil numbers $\pi$ and $\pi^{\prime}$ are said to be \emph{conjugate} if they have the same minimal polynomial over $\Q.$
\end{definition}

Now, we introduce the main result of this section, which allows us to relate $q$-Weil numbers to simple abelian varieties over $\F_q$.

\begin{theorem}[{\cite[Main Theorem]{6}} or {\cite[\S16.5]{3}}]\label{thm HondaTata}
 For every $q$-Weil number $\pi$, there exists a simple abelian variety $X$ over $\F_q$ such that $\pi_X$ is conjugate to $\pi$. Moreover, we have a bijection between the set of isogeny classes of simple abelian varieties over $\F_q$ and the set of conjugacy classes of $q$-Weil numbers given by $X \mapsto \pi_X$.
\end{theorem}
The inverse of the map $X \mapsto \pi_X$ associates to a $q$-Weil number $\pi$ a simple abelian variety $X$ over $\F_q$ such that $f_X$ is a power of the minimal polynomial $f_{\Q}^{\pi}$ of $\pi$ over $\Q.$

\section{Finite subgroups of division algebras}\label{findiv}
Recall that if $X$ is a simple abelian variety of dimension $4$ over a finite field $k$, then $\textrm{End}_k^0(X)$ is a CM-field of degree $8$ over $\Q$ or a central simple division algebra of degree $2$ (resp.\ $4$) over a quartic CM-field (resp.\ an imaginary quadratic field) by Lemma \ref{poss end alg}. Hence, in this section, we give a classification of all possible finite groups that can be embedded in the multiplicative subgroup of a division algebra with certain properties that are related to our situation. Our main reference is a paper of Amitsur~\cite{1}. We start with the following

\begin{definition}\label{def 9}
  Let $m, r$ be two relatively prime positive integers, and we put $s:=\gcd(r-1, m)$ and $t:=\frac{m}{s}.$ Also, let $n$ be the smallest integer such that $r^n \equiv 1~(\textrm{mod}~m).$ We denote by $G_{m,r}$ the group generated by two elements $a,b$ satisfying the relations
  \begin{equation*}
    a^m=1,~b^n=a^t,~ bab^{-1}=a^r.
  \end{equation*}
  This type of groups includes the \emph{dicyclic group} of order $mn,$ in which case, we often write $\textrm{Dic}_{mn}$ for $G_{m,r}$. As a convention, if $r=1,$ then we put $n=s=1,$ and hence, $G_{m,1}$ is a cyclic group of order $m.$
\end{definition}

  Given $m,r,s,t,n,$ as above, we will consider the following two conditions in the sequel: \\
  (C1) $\gcd(n,t)=\gcd(s,t)=1.$ \\
  (C2) $n=2n^{\prime}, m=2^{\alpha} m^{\prime}, s=2 s^{\prime}$ where $\alpha \geq 2,$ and $n^{\prime}, m^{\prime}, s^{\prime}$ are all odd integers. Moreover, $\gcd(n,t)=\gcd(s,t)=2$ and $r \equiv -1~(\textrm{mod}~2^{\alpha}).$ \\
\indent Now, let $p$ be a prime number that divides $m.$ We define: \\
(i) $\alpha_p$ is the largest integer such that $p^{\alpha_p}~|~m.$ \\
(ii) $n_p$ is the smallest integer satisfying $r^{n_p} \equiv 1~(\textrm{mod}~mp^{-\alpha_p}).$ \\
(iii) $\delta_p$ is the smallest integer satisfying $p^{\delta_p} \equiv 1~(\textrm{mod}~mp^{-\alpha_p}).$\\
\indent Then we have the following result that provides us with a useful criterion for a group $G_{m,r}$ to be embedded in a division ring:
\begin{theorem}[{\cite[Theorems 3, 4 and Lemma 10]{1}}]\label{thm 11}
  A group $G_{m,r}$ can be embedded in a division ring if and only if either (C1) or (C2) holds, and one of the following conditions holds: \\
  (1) $n=s=2$ and $r \equiv -1~(\textrm{mod}~m)$. \\
  (2) For every prime $q~|~n,$ there exists a prime $p~|~m$ such that $q \nmid n_p$ and that either \\
  (a) $p \ne 2$, and $\gcd(q,(p^{\delta_p}-1)/s)=1$, or \\
  (b) $p=q=2,$ (C2) holds, and $m/4 \equiv \delta_p \equiv 1~(\textrm{mod}~2).$
\end{theorem}

Now, let $G$ be a finite group. One of our main tools in this section is the following
\begin{theorem}[{\cite[Theorem 7]{1}}]\label{thm 12}
  $G$ can be embedded in a division ring if and only if $G$ is of one of the following types: \\
  (1) Cyclic groups. \\
  (2) $G_{m,r}$ where the integers $m,r,$ etc, satisfy Theorem \ref{thm 11} (which is not cyclic). \\
  (3) $\mathfrak{T}^* \times G_{m,r}$ where $\mathfrak{T}^*$ is the binary tetrahedral group of order $24$, and $G_{m,r}$ is either cyclic of order $m$ with $\gcd(m,6)=1$, or of the type (2) with $\gcd(|G_{m,r}|, 6)=1.$ In both cases, for all primes $p~|~m,$ the smallest integer $\gamma_p$ satisfying $2^{\gamma_p} \equiv 1~(\textrm{mod}~p)$ is odd. \\
  (4) $\mathfrak{O}^*$, the binary octahedral group of order $48$. \\
  (5) $\mathfrak{I}^*,$ the binary icosahedral group of order $120.$
\end{theorem}

Having stated most of the necessary results, we can proceed to achieve our goal of this section. First, we give three important lemmas, all of whose proofs follow from Theorem \ref{thm 11} unless otherwise stated.

\begin{lemma}\label{lem 15}
  Suppose that $n=2,$ $m \geq 2$ is an integer with $\varphi(m)~|~8$ and $\gcd(n,t)=\gcd(s,t)=1.$ Then the group $G:=G_{m,r}$ can be embedded in a division ring if and only if $G$ is one of the following groups (up to isomorphism): \\
  (1) $C_4$\footnote{If we want $G_{m,r}$ to be non-cyclic as in Theorem \ref{thm 12}, then we may exclude this case.}; \\
  (2) $\textrm{Dic}_{12}$, a dicyclic group of order $12$; \\
  (3) $\textrm{Dic}_{20}$, a dicyclic group of order $20$; \\
  (4) $C_5 \rtimes C_8$, a semidirect product of $C_5$ and $C_8$;\\
  (5) $C_3 \rtimes C_{16}$, a semidirect product of $C_3$ and $C_{16}$;\\
  (6) $\textrm{Dic}_{60}$, a dicyclic group of order $60$.
\end{lemma}
\begin{proof}
 Note that $t$ is odd and $m \in \{2,3,4,5,6,8,10,12,15,16,20,24,30\}.$ Hence, we have the following three cases to consider: \\
\indent (1) If $m \in \{3,5,15 \},$ then $|G|=2m$ is square-free. Now, if $m=3$ or $5$, then it is easy to see that $r \equiv -1~(\textrm{mod}~m)$ (by the definition of $n$). Similarly, if $m=15,$ then it is easy to see that $r \equiv -1$ or $r \equiv \pm 4~(\textrm{mod}~m)$. Then in any of these cases, by looking at the presentation of $G,$ we can see that $G$ is not cyclic. Hence by \cite[Corollary 5]{1}, $G$ cannot be embedded in a division ring. \\
\indent (2) If $m \in \{2,4,6,8,10,12\},$ then we have the following six subcases to consider: \\
  (i) If $m=2,$ then since $\gcd(n,t)=1,$ we get $s=2, t=1.$ In particular, we have $n=s=2$ and $r \equiv -1~(\textrm{mod}~2).$ Hence, $G$ can be $G_{2,r}=C_4.$ \\
  (ii) If $m=4,$ then since $\gcd(n,t)=1,$ we get $s=4, t=1.$ By the definition of $s$, we have $m=4~|~r-1$, and hence, this case cannot occur because of our assumption that $n=2.$ \\
\indent By a similar argument as in (ii), we will exclude all the cases when $t=1$ for the rest of the proof of the lemma.\\
  (iii) If $m=6,$ then we only need to consider the case when $s=2, t=3.$ Since $n=2,$ we have $r \equiv -1~(\textrm{mod}~6).$ Hence, $G$ can be $G_{6,r}=\textrm{Dic}_{12}.$ \\
  (iv) If $m=8,$ then we get $s=8, t=1$, and hence, this case cannot occur. \\
  (v) If $m=10,$ then we only need to consider the case when $s=2, t=5.$ Since $n=2,$ we have $r \equiv -1~(\textrm{mod}~10).$ Hence, $G$ can be $G_{10,r}=\textrm{Dic}_{20}.$ \\
  (vi) If $m=12,$ then we only need to consider the case when $s=4, t=3.$ If there exists a prime number $p~|~12$ such that $q=2~\nmid~n_p,$ then either $p=2$ or $p=3.$ Since (C2) does not hold, $p \ne 2.$ If $p=3,$ then $\alpha_3 = 1, \delta_3 = 2, $ and hence, we have $\gcd(2, (3^2 -1)/4)=2.$ Therefore this case cannot occur by Theorem \ref{thm 11}. \\
(3) If $m \in \{16,20,24,30\},$ then we have the following four subcases to consider: \\
  (i) If $m=16,$ then we get $s=16, t=1$, and hence, this case cannot occur. \\
  (ii) If $m=20,$ then we only need to consider the case when $s=4, t=5.$ Since $n=2,$ we have $r \equiv -1$ or $r \equiv \pm 9~(\textrm{mod}~20).$ By the definition of $s,$ it follows that $r \equiv 9~(\textrm{mod}~20).$ Now, by taking $q=2, p=5$ (so that $n_5=\alpha_5 =\delta_5=1$) in Theorem \ref{thm 11}, we can see that $G$ can be $G_{20,r}=C_5 \rtimes C_8$ (by looking at the presentation of $G$). \\
  (iii) If $m=24,$ then we only need to consider the case when $s=8, t=3.$ Since $n=2,$ we have $r \equiv -1$ or $r \equiv \pm 5$ or $r \equiv \pm 7$ or $r \equiv \pm 11~(\textrm{mod}~24).$ By the definition of $s,$ it follows that $r \equiv 17~(\textrm{mod}~24).$ Now, by taking $q=2, p=3$ (so that $n_3=\alpha_3 =1, \delta_3=2$) in Theorem \ref{thm 11}, we can see that $G$ can be $G_{24,r}=C_3 \rtimes C_{16}$ (by looking at the presentation of $G$). \\
  (iv) If $m=30,$ then we get $s=2, t=15$ or $s=6, t=5$ or $s=10, t=3$. If $s=2, t=15,$ then since $n=2,$ we have $r \equiv -1$ or $r \equiv \pm 11~(\textrm{mod}~30).$ By the definition of $s,$ it follows that $r \equiv -1~(\textrm{mod}~30).$ Hence, $G$ can be $G_{30,r}=\textrm{Dic}_{60}.$ If $s=6, t=5$ or $s=10, t=3,$ then by a similar argument as in (2)-(vi), we can see that both of these cases cannot occur. \\
\indent This completes the proof.
\end{proof}
\begin{remark}
The group $C_5 \rtimes C_8$ (resp.\ $C_3 \rtimes C_{16}$) is (isomorphic to) the group with GAP SmallGroup ID [40,1] (resp.\ [48,1]).
\end{remark}

\begin{lemma}\label{lem 16}
  Suppose that $n=2, m=2^{\alpha} \cdot m^{\prime}$ with $\varphi(m)~|~8, s = 2 s^{\prime}, \gcd(s,t)=2,$ and $r \equiv -1~(\textrm{mod}~2^{\alpha})$ where $\alpha \geq 2$ is an integer and $m^{\prime}, s^{\prime}$ are odd integers. Then the group $G:=G_{m,r}$ can be embedded in a division ring if and only if $G$ is one of the following groups:\\
  (1) $Q_8,$ a quaternion group; \\
  (2) $\textrm{Dic}_{16},$ a dicyclic group of order $16$; \\
  (3) $\textrm{Dic}_{24},$ a dicyclic group of order $24;$ \\
  (4) $\textrm{Dic}_{32},$ a dicyclic group of order $32;$ \\
  (5) $\textrm{Dic}_{40},$ a dicyclic group of order $40;$ \\
  (6) $\textrm{Dic}_{48},$ a dicyclic group of order $48.$
\end{lemma}
\begin{proof}
Note that $m \in \{4,8,12,16,20,24\}.$ Hence, we have the following six cases to consider: \\
\indent (1) If $m=4$, then we get $s=t=2.$ Since $n=s=2$ and $r \equiv -1~(\textrm{mod}~4)$ (by assumption), $G$ can be $G_{4,r}=Q_8$. \\
\indent (2) If $m=8,$ then we get $s=2,t=4.$ Since $n=s=2$ and $r \equiv -1~(\textrm{mod}~8)$ (by assumption), $G$ can be $G_{8,r}=\textrm{Dic}_{16}.$\\
\indent (3) If $m=12,$ then we get $s=2,t=6$ or $s=6,t=2.$ If $s=2,t=6,$ then since $n=s=2,$ we have $r \equiv -1~ (\textrm{mod}~12)$ (by the definition of $s$). Hence, $G$ can be $G_{12,r}=\textrm{Dic}_{24}.$ If $s=6,t=2,$ then since $n=2,s=6,$ we have $r \equiv 7 ~(\textrm{mod}~12).$ If there exists a prime number $p~|~12$ such that $q=2 \nmid n_p,$ then we get $p=2$. (Note that $n_2=1$ and $n_3=2.$) Then we further compute $\alpha_2=\delta_2=2$, and this contradicts Theorem \ref{thm 11}-(2). Hence, this case cannot occur. \\
\indent (4) If $m=16,$ then we get $s=2,t=8.$ Then by a similar argument as in (2), $G$ can be $G_{16,r}=\textrm{Dic}_{32}.$ \\
\indent (5) If $m=20,$ then we get $s=2,t=10$ or $s=10,t=2.$ If $s=2,t=10,$ then since $n=s=2,$ we have $r \equiv -1~ (\textrm{mod}~20)$ (by the definition of $s$). Hence, $G$ can be $G_{20,r}=\textrm{Dic}_{40}.$ If $s=10,t=2,$ then since $n=2,s=10,$ we have $r \equiv 11 ~(\textrm{mod}~20).$ If there exists a prime number $p~|~20$ such that $q=2 \nmid n_p,$ then we get $p=2$. (Note that $n_2=1$ and $n_5=2.$) Then we further compute $\alpha_2=2,\delta_2=4$, and this contradicts Theorem \ref{thm 11}-(2). Hence, this case cannot occur. \\
\indent (6) If $m=24,$ then we get $s=2,t=12$ or $s=6,t=4.$ If $s=2,t=12,$ then since $n=s=2$ and $r \equiv -1 ~(\textrm{mod}~8),$ we have $r \equiv -1~ (\textrm{mod}~24).$ Hence, $G$ can be $G_{24,r}=\textrm{Dic}_{48}.$ If $s=6,t=4,$ then since $n=2,s=6$ and $r \equiv -1 ~(\textrm{mod}~8),$ we have $r \equiv 7 ~(\textrm{mod}~24).$ Then by a similar argument as in (5), we can see that this case cannot occur. \\
\indent This completes the proof.
\end{proof}

\begin{lemma}\label{addlem 15}
  Suppose that $n=4,$ $m \geq 2$ is an integer with $\varphi(m)~|~8,$ and $\gcd(n,t)=\gcd(s,t)=1.$ Then the group $G:=G_{m,r}$ can be embedded in a division ring if and only if $G=C_5 \rtimes C_{16}$, a semidirect product of $C_5$ and $C_{16}$.
\end{lemma}
\begin{proof}
 Note that $t$ is odd and $m \in \{2,3,4,5,6,8,10,12,15,16,20,24,30\}.$ Furthermore, by the definition of $n$ and the fact that $n^2~|~|G|$ (see \cite[\S7]{1}), it suffices to consider the following two cases: \\
\indent (1) If $m=16,$ then since $\gcd(n,t)=1,$ we get $s=16, t=1.$ By the definition of $s$, we have $m=16~|~r-1$, and hence, this case cannot occur because of our assumption that $n=4.$ \\ 
\indent By a similar argument as in (1), we will exclude the case when $t=1$ for part (2) below:  \\
\indent (2) If $m=20,$ then by the definition of $n$ and $s,$ we get $r \equiv 13$ or $r \equiv 17~(\textrm{mod}~20)$ and $s=4,t=5$. In both cases (on $r$)\footnote{In fact, in this specific situation, the groups $G_{20,13}$ and $G_{20,17}$, a priori, are isomorphic to each other.}, by taking $q=2, p=5$ (so that $n_5=\alpha_5 =\delta_5=1$) in Theorem \ref{thm 11}, we can see that $G$ can be $G_{20,r}=C_5 \rtimes C_{16}$ (by looking at the presentation of $G$).\\
\indent This completes the proof.
\end{proof}
\begin{remark}
The group $C_5 \rtimes C_{16}$ is (isomorphic to) the group with GAP SmallGroup ID [80,3].
\end{remark}

Now, we focus on our situation of Lemma \ref{poss end alg}-(1) and (2) respectively. First, let $D$ be a division algebra of degree $2$ over its center $K$, where $K$ is a quartic CM-field. If $G$ is a finite subgroup of the multiplicative subgroup of $D,$ then $G$ can be of any of the types (1)-(5) in Theorem \ref{thm 12}, a priori. If the group $G:=G_{m,r}$ is contained in $D^{\times},$ then $n~|~2$ (see \cite[\S7]{1}), and hence, we have that either $n=1$ or $n=2.$ Furthermore, if $n=2,$ then $4$ divides $|G|$ (see \cite[\S7]{1}). Also, since $G$ contains an element of order $m$ and $\dim_{\Q}D = 16,$ we have $m \leq 30$ with $\varphi(m)~|~8.$ The last preliminary result that we need is the following

\begin{theorem}[{\cite[Theorems 9 and 10]{1}}]\label{exc gps}
Let $D$ be a division algebra of degree $2$ over its center $K$, where $K$ is a quartic CM-field $K$. \\
(a) If $D$ contains an $\mathfrak{O}^*,$ then $\sqrt{2} \in K.$\\
(b) If $D$ contains an $\mathfrak{I}^*,$ then $\sqrt{5} \in K.$  \\
Moreover, in both cases, we have $D = D_{2,\infty} \otimes_{\Q} K$.
\end{theorem}

Summarizing, we have the following

\begin{theorem}\label{deg2 case}
Let $D$ be a division algebra of degree $2$ over its center $K$, where $K$ is a quartic CM-field $K$. If a finite group $G$ (of even order\footnote{This assumption can be made based on the goal of this paper.}) can be embedded in $D^{\times}$, then $G$ is one of the following groups:\\
(1) $C_2, C_4, C_6, C_8, C_{10}, C_{12}, C_{16}, C_{20}, C_{24}, C_{30}$; \\
(2) $Q_8, \textrm{Dic}_{12}, \textrm{Dic}_{16}, \textrm{Dic}_{20}, \textrm{Dic}_{24}, \textrm{Dic}_{32}, \textrm{Dic}_{40}, \textrm{Dic}_{48}, \textrm{Dic}_{60}$; \\
(3) $C_5 \rtimes C_8, C_3 \rtimes C_{16}$;\\ 
(4) $\mathfrak{T}^*$; \\
(5) $\mathfrak{O}^*$; \\
(6) $\mathfrak{I}^*.$
\end{theorem}
\begin{proof}
  We refer the list of possible such groups to Theorem \ref{thm 12}. Suppose that $G$ is cyclic. Then we can write $G=\langle f\rangle$ for some element $f$ of order $d$. Then according to the argument given before Theorem \ref{exc gps}, we have that $d \in \{2,4,6,8,10,12, 16, 20, 24, 30\}$. Hence, we obtain (1). If $G=G_{m,r}$ with $n=2$ and $\phi(m) ~|~ 8,$ then (2) and (3) follow from Lemmas \ref{lem 15} and \ref{lem 16}. (If $n=1,$ then $s=m$ so that $t=1.$ Now, the presentation of the group tells us that the group is cyclic of order $m$ in this case.) Now, if $G=\mathfrak{T}^* \times G_{m,r}$ is a general $T$-group, then the only possible such groups are $\mathfrak{T}^*$ and $\mathfrak{T}^* \times C_5.$ (If $n=2$ so that $G_{m,r}$ is not cyclic, then $|G_{m,r}|=2m,$ and hence, we have $\gcd(|G_{m,r}|,6) \ne 1.$) In the latter case, we have $\gamma_5 =4$ which contradicts the condition of Theorem \ref{thm 12}. Also, clearly, $\mathfrak{T}^*$ satisfies the condition of Theorem \ref{thm 12}, and hence, we get (4). Finally, the other possibilities for $G$ are $\mathfrak{O}^*$ and $\mathfrak{I}^*.$ If $G=\mathfrak{O}^*,$ then $|G|=48$, and in this case, $\sqrt{2} \in K$ by Theorem \ref{exc gps}. If $G=\mathfrak{I}^*,$ then $|G|=120,$ and in this case, $\sqrt{5} \in K$ by Theorem \ref{exc gps}. \\
\indent This completes the proof.
\end{proof}

By a similar argument as above, we can deal with the other case.
\begin{theorem}\label{deg4 case}
Let $D$ be a division algebra of degree $4$ over its center $K$, where $K$ is an imaginary quadratic field. If a finite group $G$ (of even order) can be embedded in $D^{\times}$, then $G$ is one of the following groups: \\
(1) $C_2, C_4, C_6, C_8, C_{10}, C_{12}, C_{16}, C_{20}, C_{24}, C_{30}$; \\
(2) $\textrm{Dic}_{12}, \textrm{Dic}_{20}, C_5 \rtimes C_8, C_3 \rtimes C_{16}, \textrm{Dic}_{60}$;\\
(3) $C_5 \rtimes C_{16}$. 
\end{theorem}
\begin{proof}
Since $K$ is an algebraic number field, it follows from \cite[Theorems 2 and 9]{1} that all Sylow subgroups of $G$ are cyclic\footnote{In other words, $G$ is of type (2A) in the sense of \cite[Theorem 2]{1}.}. Now, since the Sylow $2$-subgroup of any $\mathfrak{T}^* \times G_{m,r}$ of Theorem \ref{thm 12}-(3) is $Q_8,$ we can exclude these groups. If $G=\mathfrak{O}^*$ (resp.\ $G=\mathfrak{I}^*$), then we have that $\Q(\sqrt{2}) \subseteq K$ (resp.\ $\Q(\sqrt{5}) \subseteq K$) by Theorem \ref{exc gps}, which contradicts the fact that $K$ is an imaginary quadratic field. If $G$ is a cyclic group of order $d,$ then we have $\varphi(d)~|~8$, and hence, by a similar argument as in the proof of Theorem \ref{deg2 case}, we obtain (1). Finally, if $G:=G_{m,r}$ for some relatively prime integers $m,r,$ then we obtain (2) and (3) by Lemmas \ref{lem 15} and \ref{addlem 15}. \\
\indent This completes the proof.  
\end{proof}

Regarding our goal of this paper, the above theorems have a nice consequence.
\begin{corollary}\label{cor 19}
  Let $X$ be a simple abelian variety of dimension $4$ over a finite field $k.$ Let $G$ be a finite subgroup of the multiplicative subgroup of $\textrm{End}_k^0(X).$ Then $G$ is one of the following groups: \\
(1) $C_2, C_4, C_6, C_8, C_{10}, C_{12}, C_{16}, C_{20}, C_{24}, C_{30}$; \\
(2) $Q_8, \textrm{Dic}_{12}, \textrm{Dic}_{16}, \textrm{Dic}_{20}, \textrm{Dic}_{24}, \textrm{Dic}_{32}, \textrm{Dic}_{40}, \textrm{Dic}_{48}, \textrm{Dic}_{60}$; \\
(3) $C_5 \rtimes C_8, C_3 \rtimes C_{16}$;\\
(4) $C_5 \rtimes C_{16}$;\\ 
(5) $\mathfrak{T}^*$; \\
(6) $\mathfrak{O}^*$; \\
(7) $\mathfrak{I}^*.$
\end{corollary}

\begin{proof}
  This follows from Lemma \ref{poss end alg} and Theorems \ref{deg2 case}, \ref{deg4 case}.
\end{proof}

\section{Main Result}\label{main}
In this section, we give a classification of finite groups that can be realized as the automorphism group of a simple polarized abelian variety of dimension $4$ over a finite field which is maximal in the following (slightly more general) sense: let $g \geq 1$ be an integer.

\begin{definition}[{\cite[Definition 6.1]{8}}]\label{def 20}
  Let $X$ be an abelian variety of dimension $g$ over a field $k$, and let $G$ be a finite group. Suppose that the following two conditions hold: \\
\indent (i) there exists an abelian variety $X^{\prime}$ over $k$ that is $k$-isogenous to $X$ with a polarization $\mathcal{L}$ such that $G=\textrm{Aut}_k(X^{\prime},\mathcal{L}),$ and \\
\indent (ii) there is no finite group $H$ such that $G$ is isomorphic to a proper subgroup of $H$ and $H=\textrm{Aut}_k (Y,\mathcal{M})$ for some abelian variety $Y$ over $k$ that is $k$-isogenous to $X$ with a polarization $\mathcal{M}.$ \\
\indent In this case, $G$ is said to be \emph{realizable maximally (or maximal, in short) in the isogeny class of $X$} as the full automorphism group of a polarized abelian variety of dimension $g$ over $k$.
\end{definition}
In our case, we take $g=4,$ $k$ a finite field, and $X$ to be simple. \\

Now, to introduce our main result, we need the following lemmas, which are obtained by examining some of the situations of Corollary \ref{cor 19} further.
\begin{lemma}\label{dic not lem0}
Let $G=\textrm{Dic}_{20}.$ Then there exists no simple abelian variety $X$ of dimension $4$ over a finite field $k$ such that $G$ is a finite subgroup of the multiplicative subgroup of $\textrm{End}_k^0(X).$
\end{lemma}
\begin{proof}
Let $G=\textrm{Dic}_{20}$ and suppose on the contrary that there exists such an abelian variety $X$ of dimension $4$ over some finite field $k.$ Let $D=\textrm{End}_k^0(X).$ Note that since $G \leq D^{\times},$ we have $D_{2,\infty} =\left(\frac{-1,-1}{\Q} \right) \subset D_{2,\infty} \otimes_{\Q} \Q(\sqrt{5}) \subset D$. 
Also, by Lemma \ref{poss end alg} and Theorems \ref{deg2 case}, \ref{deg4 case}, we only need to consider the following two cases: \\
\indent (1) $D$ is a quaternion division algebra over $K$ where $K:=\Q(\pi_X)$ is a quartic CM-field. Then we have $D_{2,\infty} \otimes_{\Q} K =D$ (by dimension counting). Also, we note that $L:=\Q(\zeta_{10}) \subset D$ because $C_{10} \leq G \leq D^{\times}.$ Let $B=C_{D}(L)$ (resp.\ $C=C_{D}(B)$) be the centralizer of $L$ (resp.\ $B$) in $D.$ By definition, $B$ (resp.\ $C$) contains both $K$ and $L.$ In particular, $B$ (resp.\ $C$) is a $K$-subalgebra of $D$ containing $L.$ By the double centralizer theorem, we have that $[B:K]\cdot [C:K]=4.$ If either $[B:K]=1$ or $[B:K]=4$ (so that $[C:K]=1$), then we can see that $K=L.$ (For this case, we refer to subcase (i) below.) If $[B:K]=[C:K]=2,$ then we get that $B$ is commutative, and hence, $B$ is a maximal subfield of $D$ with $[B:\Q]=8.$ Now, suppose that $K \cap L = \Q.$ Let $M=K.L$ be the compositum of $K$ and $L$ in $B$. In particular, we have $M \subseteq B.$ But then since $K \cap L = \Q,$ we also have that $[M:\Q]=[K:\Q]\cdot [L:\Q]=16,$ which is a contradiction. Consequently, we only need to consider the following two cases further: \\
(i) If $K = L = \Q(\zeta_{10}),$ then since there is a unique prime $\nu$ of $K$ lying over $2$ so that $2~|~[K_{\nu}:\Q_2]=4$, and $\Q(\zeta_{10})$ is a CM-field, it follows that $D$ splits at all primes of $K$, which contradicts the fact that $D$ is a division algebra.   \\
(ii) If $K \cap L$ is a quadratic number field, then we have $K \cap L = \Q(\sqrt{5}).$ Note that $2$ is inert in $\Q(\sqrt{5})$. Then since $2 ~|~ [K_{\nu}:\Q_2]$ for any prime $\nu$ of $K$ lying over $2,$ and $K$ is a CM-field, it follows again that $D$ splits at all primes of $K,$ which is a contradiction. \\
\indent Thus we can see that this case cannot occur. \\
\indent (2) $D$ is a central simple division algebra of degree $4$ over $K$ where $K:=\Q(\pi_X)$ is an imaginary quadratic field. Then let $B=D_{2,\infty} \otimes_{\Q} K$, and let $C=C_{D}(B)$ be the centralizer of $B$ in $D.$ In particular, the center of both $B$ and $C$ is $K.$ By the double centralizer theorem, we have $[C:K]=4$ and $B \otimes_K C \cong D.$ Then since $D$ has period $4$ (being a central simple division algebra of degree $4$ over a number field $K$) while the period of $B \otimes_K C$ is at most $2$ (\cite[$\S$1.5]{4}), this is a contradiction. Thus we can see that this case cannot occur. \\
\indent From (1) and (2), the desired result follows. \\
\indent This completes the proof.
\end{proof}

By a similar argument, we obtain other useful results.
\begin{lemma}\label{dic not lem}
Let $G =\textrm{Dic}_{16}$ or $\textrm{Dic}_{24}.$ Then there exists no simple abelian variety $X$ of dimension $4$ over a finite field $k$ such that $G$ is a finite subgroup of the multiplicative subgroup of $\textrm{End}_k^0(X).$
\end{lemma}
\begin{proof}
Let $G=\textrm{Dic}_{16}$ (resp.\ $G=\textrm{Dic}_{24}$) and suppose on the contrary that there exists such an abelian variety of dimension $4$ over some finite field $k.$ Since the 2-Sylow subgroup of $G$ is a generalized quaternion group of order $16$ (resp.\ a quaternion group of order $8$), it follows from Lemma \ref{poss end alg} and \cite[Theorem 9]{1} that $D:=\textrm{End}_k^0(X)=D_{2,\infty}\otimes_{\Q} K$ is a quaternion division algebra over $K$, where $K:=\Q(\pi_X)$ is a quartic CM-field. Also, since $C_{8} \leq \textrm{Dic}_{16}$ (resp.\ $C_{12} \leq \textrm{Dic}_{24}$), we have that $L:=\Q(\zeta_{8}) \subset D$ (resp. $L:=\Q(\zeta_{12})$). Then we can proceed as in the proof of Lemma \ref{dic not lem0} for the case (1) to see that this cannot occur. \\
\indent This completes the proof.
\end{proof}

\begin{lemma}\label{Q8 not max lem}
Let $G = Q_8$ or $ \mathfrak{T}^*.$ Then there exists no simple abelian variety $X$ of dimension $4$ over a finite field $k$ such that $G$ is a finite subgroup of the multiplicative subgroup of $\textrm{End}_k^0(X).$
\end{lemma}
\begin{proof}
Let $G=Q_8$ and suppose on the contrary that there exists such an abelian variety of dimension $4$ over some finite field $k:=\F_q.$ Since the 2-Sylow subgroup of $G$ is a quaternion group of order $8$, we can see that $D:=\textrm{End}_k^0(X)=D_{2,\infty}\otimes_{\Q} K$ is a quaternion division algebra over $K$, where $K:=\Q(\pi_X)$ is a quartic CM-field, as before. This implies that $q=2^a$ for some $a \geq 1.$ We examine the ramification behavior of the prime $2$ in $K,$ together with that of the quaternion division algebra $D_{2,\infty} \otimes_{\Q} K.$ First of all, since $D_{2,\infty} \otimes_{\Q} K$ must be a division algebra in our case, we know that the number of primes in $K$ above $2$ is at least $3$. Thus we consider the following two cases:   \\
\indent (1) Suppose that there are three primes $\nu_1, \nu_2, \nu_3$ of $K$ lying above $2.$ Then there is one prime, say $\nu_3$, at which $D_{2,\infty}$ splits, and we consider the following two subcases: \\
(i) $2 =\nu_1 \nu_2 \nu_3^2$ in $K$ so that we have $\pi_X \cdot \overline{\pi_X} = q = \nu_1^a \nu_2^a \nu_3^{2a}$ in $K.$ Assume that $\pi_X = \nu_1^i \nu_2^{a-i} \nu_3^j$ for some $0 \leq i \leq a$ and $0 \leq j \leq 2a$. In particular, the local invariants of $D$ at those primes are given by $i/a, (a-i)/a,$ and $2j/a.$ Then since we know that $D$ is a quaternion algebra over $K,$ which is ramified precisely at $\nu_1$ and $\nu_2,$ it follows from Proposition \ref{local inv}-(b) that $a$ must be even, $i=a/2$ and $j=a$. Then we get that $\pi_X = \overline{\pi_X}$, and it follows from \cite[Proposition 3.2]{Gon(1998)} and \cite[Lemma 2.1]{Xue(2016)} that $K=\Q(\zeta_n)$ for $n \in \{5, 8, 12\},$ which contradicts the ramification behavior of $2$ in $K.$ Thus this case cannot occur. \\
(ii) $2 =\nu_1 \nu_2 \nu_3$ in $K$ so that we have $\pi_X \cdot \overline{\pi_X} = q = \nu_1^a \nu_2^a \nu_3^a$ in $K.$ Then we can proceed as in the case (i) to see that this case cannot occur, either.         \\
\indent (2) Suppose that $2$ splits completely in $K$ i.e.\ there are four primes $\nu_1, \nu_2, \nu_3, \nu_4$ of $K$ lying above $2.$ (In particular, $D_{2,\infty} \otimes_{\Q} K$ is ramified at those four primes of $K.$) Then we have $2=\nu_1  \nu_2 \nu_3 \nu_4$ in $K$ so that $\pi_X \cdot \overline{\pi_X} = q = \nu_1^a \nu_2^a \nu_3^a \nu_4^a.$ Assume that $\pi_X = \nu_1^i \nu_2^{a-i} \nu_3^j \nu_4^{a-j}$ for some $0 \leq i,j \leq a.$ In particular, the local invariants of $D$ at those primes are given by $i/a, (a-i)/a, j/a$ and $(a-j)/a.$ Then since we know that $D$ is a quaternion algebra over $K$, which is ramified at $\nu_1, \nu_2, \nu_3,$ and $\nu_4,$ it follows from Proposition \ref{local inv}-(b) that $a$ must be even and $i=j=a/2$. Then we can proceed as in the case (1)-(i) to see that this case cannot occur. \\
\indent Hence the desired result follows from (1) and (2). \\
For $G=\mathfrak{T}^*$, we can use a similar argument as above, or use the fact that $Q_8 \leq \mathfrak{T}^*.$ \\
\indent This completes the proof.
\end{proof}

\begin{lemma}\label{O,I not lem}
Let $G = \mathfrak{O}^*$ or $\mathfrak{I}^*.$ Then there exists no simple abelian variety $X$ of dimension $4$ over a finite field $k$ such that $G$ is a finite subgroup of the multiplicative subgroup of $\textrm{End}_k^0(X).$
\end{lemma}
\begin{proof}
Let $G=\mathfrak{O}^*$ (resp.\ $G=\mathfrak{I}^*$) and suppose on the contrary that there exists such an abelian variety of dimension $4$ over some finite field $k.$ Since the 2-Sylow subgroup of $G$ is a generalized quaternion group of order $16$ (resp.\ a quaternion group of order $8$), we can see that $D:=\textrm{End}_k^0(X)=D_{2,\infty}\otimes_{\Q} K$ is a quaternion division algebra over $K$, where $K:=\Q(\pi_X)$ is a quartic CM-field, as before. Furthermore, by Theorem \ref{exc gps}, we know that $\Q(\sqrt{2}) \subset K$ (resp.\ $\Q(\sqrt{5}) \subset K$). Since $2$ is totally ramified in $\Q(\sqrt{2})$ (resp.\ $2$ is inert in $\Q(\sqrt{5})$), we can see that $2 ~|~ [K_{\nu}:\Q_2]$ for any prime $\nu$ of $K$ lying over $2,$ and hence, it follows that $D$ splits at all primes of $K$, which contradicts the fact that $D$ is a division algebra. Thus this case cannot occur. \\
\indent This completes the proof.
\end{proof}

The following lemma deals with a special class of groups in light of Remark \ref{nebe rem} below.
\begin{lemma}\label{Nebe not lem}
Let $G \in \{\textrm{Dic}_{32}, \textrm{Dic}_{40}, \textrm{Dic}_{48}, \textrm{Dic}_{60}\}$. Then there exists no simple abelian variety $X$ of dimension $4$ over a finite field $k$ such that $G$ is a finite subgroup of the multiplicative subgroup of $\textrm{End}_k^0(X).$
\end{lemma}
\begin{proof}
Let $G=\textrm{Dic}_{48}$ and suppose on the contrary that there exists such an abelian variety of dimension $4$ over some finite field $k.$ Since the 2-Sylow subgroup of $G$ is a generalized quaternion group of order $16$, we can see that $D:=\textrm{End}_k^0(X)=D_{2,\infty}\otimes_{\Q} K$ is a quaternion division algebra over $K$, where $K:=\Q(\pi_X)$ is a quartic CM-field, as before. Also, since $C_{24} \leq \textrm{Dic}_{48}$, we have that $L:=\Q(\zeta_{24}) \subset D.$ Then by a similar argument as in the proof of Lemma \ref{dic not lem0}, we can see that $L$ must contain $K$, and hence, it follows that $K \in \{\Q(\zeta_8), \Q(\zeta_{12}), \Q(\sqrt{-1}, \sqrt{6}), \Q(\sqrt{-3}, \sqrt{2}), \Q(\sqrt{-2}, \sqrt{3}), \Q(\sqrt{-2}, \sqrt{6}), \Q(\zeta_{24}+\zeta_{24}^{-1})\}.$ (In fact, $K \ne \Q(\zeta_{24}+\zeta_{24}^{-1})$ because of our assumption that $K$ is a CM-field.) Now, it is easy to see that all the six CM-fields split $D_{2,\infty},$ and this contradicts the fact that $D$ is a division algebra. Thus this case cannot occur. \\ 
\indent For $G = \textrm{Dic}_{32}$ (resp.\ $\textrm{Dic}_{40}),$ we can proceed in a similar fashion with $L:=\Q(\zeta_{16})$ (resp.\ $\Q(\zeta_{20})$). \\
\indent Finally, let $G=\textrm{Dic}_{60}$. In view of Lemma \ref{poss end alg} and Theorems \ref{deg2 case}, \ref{deg4 case}, we consider the following two cases: \\ 
\indent (1) $D$ is a quaternion division algebra over $K$ where $K:=\Q(\pi_X)$ is a quartic CM-field. Since $C_{30} \leq \textrm{Dic}_{60},$ we have that $L:=\Q(\zeta_{30}) \subset D.$ Then by a similar argument as in the proof of Lemma \ref{dic not lem0}, we can see that $L$ must contain $K$ so that $K \in \{\Q(\zeta_5), \Q(\sqrt{5},\sqrt{-3}), \Q(\zeta_{15}+\zeta_{15}^{-1})\}$. Moreover, since $D^{\times}/K^{\times}$ contains $\textrm{Dic}_{60} /\{\pm 1\} \cong D_{15},$ whence, $C_{15},$ it follows from \cite[Lemma 2.1]{2} that $\zeta_{15}+\zeta_{15}^{-1} \in K$. Thus we can see that $K=\Q(\zeta_{15}+\zeta_{15}^{-1}),$ and this contradicts the fact that $K$ is a CM-field. Thus we can see that this case cannot occur. \\
\indent (2) $D$ is a central simple division algebra of degree $4$ over $K$ where $K:=\Q(\pi_X)$ is an imaginary quadratic field. Then since $\textrm{Dic}_{20} \leq \textrm{Dic}_{60}$ so that $D_{2,\infty} \subset D,$ we can proceed in a similar fashion as in the proof of Lemma \ref{dic not lem0} for the case (2) to see that this case cannot occur. \\
\indent From (1) and (2), the desired result follows for $G=\textrm{Dic}_{60}$. \\
\indent This completes the proof.
\end{proof}
\begin{remark}\label{nebe rem}
In view of \cite[Theorem 6.1]{13}, each of the four groups in Lemma \ref{Nebe not lem} is an absolutely irreducible maximal finite subgroup of some totally definite quaternion algebra over a totally real number field. 
\end{remark}

\begin{remark}
We note that Lemma \ref{O,I not lem} (resp.\ Lemma \ref{Nebe not lem}) also follows from Lemma \ref{Q8 not max lem} (resp.\ Lemmas \ref{dic not lem0} and \ref{dic not lem}) in view of the fact that $\mathfrak{T}^* \leq \mathfrak{O}^*, \mathfrak{I}^*$ (resp.\ $\textrm{Dic}_{16} \leq \textrm{Dic}_{32}$, $\textrm{Dic}_{20} \leq \textrm{Dic}_{40}, \textrm{Dic}_{60},$ and $\textrm{Dic}_{24} \leq \textrm{Dic}_{48}).$ Even though, this is the case, it could be also somewhat instructive to give a detailed proof as in the above.
\end{remark}

Finally, for the group $\textrm{Dic}_{12},$ we need to combine those arguments given in Lemmas \ref{dic not lem0} and \ref{Q8 not max lem}.
\begin{lemma}\label{dic12 not lem}
Let $G=\textrm{Dic}_{12}.$ Then there exists no simple abelian variety $X$ of dimension $4$ over a finite field $k$ such that $G$ is a finite subgroup of the multiplicative subgroup of $\textrm{End}_k^0(X).$
\end{lemma}
\begin{proof}
Let $G=\textrm{Dic}_{12}$ and suppose on the contrary that there exists such an abelian variety $X$ of dimension $4$ over some finite field $k:=\mathbb{F}_q.$ Let $D=\textrm{End}_k^0(X).$ Note that since $G \leq D^{\times},$ we have $D_{3,\infty} =\left(\frac{-1,-3}{\Q} \right) \subset D$. 
Also, by Lemma \ref{poss end alg} and Theorems \ref{deg2 case}, \ref{deg4 case}, we only need to consider the following two cases: \\
\indent (1) $D$ is a quaternion division algebra over $K$ where $K:=\Q(\pi_X)$ is a quartic CM-field. Then we have $D_{3,\infty} \otimes_{\Q} K =D$ (by dimension counting), and this implies that $q=3^a$ for some $a \geq 1$. Now, we can proceed as in the whole proof of Lemma \ref{Q8 not max lem} to see that this case cannot occur. \\
\indent (2) $D$ is a central simple division algebra of degree $4$ over $K$ where $K:=\Q(\pi_X)$ is an imaginary quadratic field. Let $B=D_{3,\infty} \otimes_{\Q} K.$ Then we can proceed as in the proof of Lemma \ref{dic not lem0} for the case (2) to see that this case cannot occur. \\
\indent This completes the proof.  
\end{proof}

Also, we record one result on the embeddability of certain number fields into a specific division algebra, which will be used in the proof of our main theorem.

\begin{lemma}\label{embed lem1}
Let $D$ be a quaternion division algebra over $\Q(\zeta_8)$, which is ramified exactly at the places $\nu$ of $\Q(\zeta_8)$ lying over $97.$ Then $\Q(\zeta_{16})$ and $\Q(\zeta_{24})$ cannot be embedded into $D$.
\end{lemma}
\begin{proof}
Let $\mathfrak{p}$ be a prime of $K:=\Q(\zeta_{8})$ lying over $97$, and let $\mathfrak{P}$ be a prime of $L:=\Q(\zeta_{16})$ lying over $\mathfrak{p}$. In particular, $\mathfrak{P}$ lies over $97.$ Since $97$ splits completely in both $K$ and $L$, it follows that $[L_{\mathfrak{P}}:K_{\mathfrak{p}}]=1.$ On the other hand, by our assumption, $D$ is ramified at $\mathfrak{p}$ i.e.\ we have $D_{\mathfrak{p}}:=D \otimes_K K_{\mathfrak{p}}$ is a quaternion division algebra over $K_{\mathfrak{p}}.$ Since $2$ does not divide $1=[L_{\mathfrak{P}}:K_{\mathfrak{p}}],$ it follows from the first theorem of \cite[p. 407]{10} that $L$ cannot be embedded into $D.$ By a similar argument, we can show that $\Q(\zeta_{24})$ cannot be embedded into $D$, either. \\
\indent This completes the proof.  
\end{proof}






\indent Now, we are ready to introduce the main theorem of this paper. 

\begin{theorem}\label{main thm}
Let $G$ be a finite group. Then there exists a finite field $k$ and a simple abelian variety $X$ of dimension $4$ over $k$ such that $G$ is the automorphism group of a simple polarized abelian variety of dimension $4$ over $k$, which is maximal in the isogeny class of $X$ if and only if $G$ is one of the following groups in Table 2 (up to isomorphism):
  \begin{center}
  \begin{tabular}{c c  }
\hline
$$ & $G$ \\
\hline
$\sharp 1$ & $C_2$  \\
\hline
$\sharp 2$ & $C_4$ \\
\hline
$\sharp 3$ & $C_6$ \\
\hline
$\sharp 4$ & $C_8$  \\
\hline
$\sharp 5$ & $C_{10}$ \\
\hline
$\sharp 6$ & $C_{12}$  \\
\hline
$\sharp 7$ & $C_{16}$  \\
\hline
$\sharp 8$ & $C_{20}$  \\
\hline
$\sharp 9$ & $C_{24}$  \\
\hline
$\sharp 10$ & $C_{30}$  \\
\hline
$\sharp 11$ & $C_5 \rtimes C_{8}$  \\
\hline
$\sharp 12$ & $C_3 \rtimes C_{16}$  \\
\hline
$\sharp 13$ & $C_5 \rtimes C_{16}$  \\
\hline
\end{tabular}
\vskip 4pt
\textnormal{Table 2} \\
\textnormal{Maximal automorphism groups of simple polarized abelian fourfolds over finite fields.}
\end{center}
\end{theorem}
\begin{proof}
  Suppose first that there exists a finite field $k$ and a simple abelian variety $X$ of dimension $4$ over $k$ such that $G$ is the automorphism group of a polarized abelian variety over $k,$ which is maximal in the isogeny class of $X$. (In particular, we have that $|G|$ is even because $-1$ preserves the polarization $\mathcal{L}$ i.e.\ $-1 \in G$.) Then by Lemma \ref{poss end alg}, Corollary \ref{cor 19}, Lemmas \ref{dic not lem0}, \ref{dic not lem}, \ref{Q8 not max lem}, \ref{O,I not lem}, \ref{Nebe not lem}, and \ref{dic12 not lem}, $G$ is one of the 13 groups in the above table. Hence, it suffices to show the converse. We prove the converse by considering them one by one. \\
\indent (1) Take $G=C_2.$ Let $k=\F_{2}$ and let $\pi$ be a zero of the polynomial $t^8 +2t^6 -t^5+t^4-2t^3 +8t^2 +16 \in \Z[t]$ so that $\pi$ is a $2$-Weil number. By Theorem \ref{thm HondaTata}, there exists a simple abelian variety $X$ over $\F_{2}$ of dimension $r$ such that $\pi_X$ is conjugate to $\pi.$ Then we have that $K:=\Q(\pi_X)=\Q(\pi)$ is a CM-field of degree $8$ so that $K$ has no real embeddings, and hence, we can compute that all the local invariants of $D:=\textrm{End}_k^0(X)$ are zero. (Note that $2$ splits into a product of two primes in $K$.) Hence it follows from Proposition \ref{local inv}-(b) and Corollary \ref{cor TateEnd0}-(b) that $D=K$ and $r=4.$ Let $\mathcal{O}$ be the ring of integers of $K.$ Since $\mathcal{O}$ is a maximal $\Z$-order in $D,$ there exists a simple abelian variety $X^{\prime}$ over $k$ such that $X^{\prime}$ is $k$-isogenous to $X$ and $\textrm{End}_k(X^{\prime})=\mathcal{O}$ by \cite[Theorem 3.13]{17}. Note also that $C_2 \leq \mathcal{O}^{\times} = \Z^3 \times C_2 =\textrm{Aut}_k(X^{\prime}).$ \\
\indent Now, let $\mathcal{L}$ be an ample line bundle on $X^{\prime}$, and put
  \begin{equation*}
    \mathcal{L}^{\prime}:=\bigotimes_{f \in C_2} f^* \mathcal{L}.
  \end{equation*}
  Then $\mathcal{L}^{\prime}$ is also an ample line bundle on $X^{\prime}$ that is preserved under the action of $C_2$ so that $C_2 \leq \textrm{Aut}_k(X^{\prime},\mathcal{L}^{\prime}).$
Note also that the maximal finite subgroup of $D^{\times}$ is $C_2$ by Dirichlet Unit Theorem. Since $\textrm{Aut}_k (X^{\prime},\mathcal{L}^{\prime})$ is a finite subgroup of $D^{\times},$ it follows that
  \begin{equation*}
    G = C_2 =\textrm{Aut}_k(X^{\prime},\mathcal{L}^{\prime}).
  \end{equation*}
Furthermore, again, since $C_2$ is a maximal finite subgroup of $D^{\times}$, we can conclude that $G$ is maximal in the isogeny class of $X.$ \\
\indent (2) Take $G=C_4.$ Let $k=\F_{241^4}$ and let $\pi$ be a zero of the quadratic polynomial $t^2 +240 \cdot 241 t +241^4 \in \Z[t]$ so that $\pi$ is a $241^4$-Weil number. By Theorem \ref{thm HondaTata}, there exists a simple abelian variety $X$ over $\F_{241^4}$ of dimension $r$ such that $\pi_X$ is conjugate to $\pi.$ Then we have $K:=\Q(\pi_X)=\Q(\pi)=\Q(\sqrt{-1})$ so that $K$ has no real embeddings. Also, note that $241$ splits completely in $K$, and we can compute the local invariants of $D:=\textrm{End}_k^0(X)$ at the two places of $K$ lying above $241$ to be $3/4$ and $1/4$. Hence it follows from Proposition \ref{local inv}-(b) and Corollary \ref{cor TateEnd0}-(b) that $D$ is a central simple division algebra of degree $4$ over $K$ and $r=4.$ Let $\mathcal{O}$ be a maximal $\Z$-order in $D$ containing the ring of integers $\Z[\sqrt{-1}]$ of $K$. Then since $\mathcal{O}$ is a maximal $\Z$-order in $D,$ there exists a simple abelian variety $X^{\prime}$ over $k$ such that $X^{\prime}$ is $k$-isogenous to $X$ and $\textrm{End}_k(X^{\prime})=\mathcal{O}$ by \cite[Theorem 3.13]{17}. Note also that $C_4 \cong \Z[\sqrt{-1}]^{\times} \leq \mathcal{O}^{\times} =\textrm{Aut}_k(X^{\prime}).$ \\
\indent Now, let $\mathcal{L}$ be an ample line bundle on $X^{\prime}$, and put
  \begin{equation*}
    \mathcal{L}^{\prime}:=\bigotimes_{f \in C_4} f^* \mathcal{L}.
  \end{equation*}
  Then $\mathcal{L}^{\prime}$ is also an ample line bundle on $X^{\prime}$ that is preserved under the action of $C_4$ so that $C_4 \leq \textrm{Aut}_k(X^{\prime},\mathcal{L}^{\prime}).$
We claim that $C_4$ is a maximal finite subgroup of $D^{\times}.$ Indeed, suppose on the contrary that there is a finite group $H \leq D^{\times}$ containing $C_4$ properly. Then by Theorem \ref{deg4 case} and Lemmas \ref{dic not lem0}, \ref{Nebe not lem}, we have that 
\begin{equation*}
H \in \{ C_8, C_{12}, C_{16}, C_{20}, C_{24}, C_5 \rtimes C_8, C_3 \rtimes C_{16}, C_5 \rtimes C_{16}\}.
\end{equation*}
Now, since $241$ splits completely in $\Q(\zeta_{n})$ for $n \in \{8,12,16,20,24\},$ it follows from the first theorem of \cite[p. 407]{10} that $\Q(\zeta_n)$ cannot be embedded in $D$, and hence, $C_n$ is not a finite subgroup of $D^{\times}$ for all $n=8, 12, 16, 20, 24.$ Also, since $C_8 \leq C_5 \rtimes C_8$ and $C_{16} \leq C_3 \rtimes C_{16}, C_5 \rtimes C_{16}$, we can see that $C_5 \rtimes C_8, C_3 \rtimes C_{16}$, and $C_5 \rtimes C_{16}$ are not finite subgroups of $D^{\times}.$ 
Thus we can conclude that $C_4$ is a maximal finite subgroup of $D^{\times}.$ Then since $\textrm{Aut}_k (X^{\prime},\mathcal{L}^{\prime})$ is a finite subgroup of $D^{\times},$ it follows that
  \begin{equation*}
    G = C_4 =\textrm{Aut}_k(X^{\prime},\mathcal{L}^{\prime}).
  \end{equation*}
Furthermore, again, since $C_4$ is a maximal finite subgroup of $D^{\times}$, we can conclude that $G$ is maximal in the isogeny class of $X.$ \\
\indent (3) For $G=C_6,$ let $k=\F_{241^4}$ and let $\pi$ be a zero of the quadratic polynomial $t^2 + 286 \cdot 241 t + 241^4 \in \Z[t].$ Then we can proceed by a similar argument as in the proof of (2) with these choices.  \\
\indent (4) Take $G=C_8.$ Let $k=\F_{9409}.$ Then $\pi:=97 \cdot \zeta_8$ is a $9409$-Weil number, and hence, by Theorem \ref{thm HondaTata}, there exists a simple abelian variety $X$ over $\F_{9409}$ of dimension $r$ such that $\pi_X$ is conjugate to $\pi.$ Then we have that $\Q(\pi_X)=\Q(\zeta_8),$ and hence, $\Q(\pi_X)$ has no real embeddings. Since $97$ splits completely in $\Q(\zeta_8),$ we can compute that $\textrm{inv}_{v}(\textrm{End}_k^0(X))=1/2$ for any place $\nu$ of $\Q(\zeta_8)$ lying over $97$ by Proposition \ref{local inv}-(a). Hence it follows from Proposition \ref{local inv}-(b) that $\textrm{End}_k^0(X)$ is a quaternion algebra over $\Q(\pi_X)$, which, in turn, implies that $r=4$ by Corollary \ref{cor TateEnd0}-(b). Now, it is well known that such $D:=\textrm{End}_k^0(X)$ has a maximal $\Z$-order $\mathcal{O}$ containing the ring of integers $\Z[\zeta_8]$ of $\Q(\zeta_8).$ Since $\mathcal{O}$ is a maximal $\Z$-order in $D,$ there exists a simple abelian variety $X^{\prime}$ over $k$ such that $X^{\prime}$ is $k$-isogenous to $X$ and $\textrm{End}_k(X^{\prime})=\mathcal{O}$ by \cite[Theorem 3.13]{17}. Note also that $C_8 \cong \langle \zeta_8 \rangle \leq \Z[\zeta_8]^{\times} \leq \mathcal{O}^{\times}=\textrm{Aut}_k(X^{\prime}).$   \\
\indent Now, let $\mathcal{L}$ be an ample line bundle on $X^{\prime}$, and put
  \begin{equation*}
    \mathcal{L}^{\prime}:=\bigotimes_{f \in \langle \zeta_8 \rangle} f^* \mathcal{L}.
  \end{equation*}
  Then $\mathcal{L}^{\prime}$ is also an ample line bundle on $X^{\prime}$ that is preserved under the action of $\langle \zeta_8 \rangle$ so that $\langle \zeta_8 \rangle \leq \textrm{Aut}_k(X^{\prime},\mathcal{L}^{\prime}).$ We claim that $C_8$ is a maximal finite subgroup of $D^{\times}.$ Indeed, suppose on the contrary that there is a finite group $H \leq D^{\times}$ containing $C_8$ properly. Then by Theorem \ref{deg2 case} and Lemmas \ref{dic not lem}, \ref{O,I not lem}, and \ref{Nebe not lem}, we have that 
\begin{equation*}
H \in \{ C_{16}, C_{24}, C_5 \rtimes C_8, C_3 \rtimes C_{16} \}.
\end{equation*}
But then, by Lemma \ref{embed lem1} (together with the fact that $C_{16} \leq C_3 \rtimes C_{16}$) and \cite[Table 5]{Inneke} (for the group $C_5 \rtimes C_8$), we can see that $H$ cannot be any of those groups in the above set.
Thus we can conclude that $C_8$ is a maximal finite subgroup of $D^{\times}.$ Then since $\textrm{Aut}_k (X^{\prime},\mathcal{L}^{\prime})$ is a finite subgroup of $D^{\times},$ it follows that
  \begin{equation*}
    G = C_8 =\textrm{Aut}_k(X^{\prime},\mathcal{L}^{\prime}).
  \end{equation*}
Furthermore, again, since $C_8$ is a maximal finite subgroup of $D^{\times}$, we can conclude that $G$ is maximal in the isogeny class of $X.$ \\
\indent (5)-(6): For $G=C_{10}$ (resp.\ $C_{12}$), let $k=\F_{3721}$ (resp.\ $k=\F_{5329}$) and let $\pi=61 \cdot \zeta_{10}$ (resp.\ $\pi=73 \cdot \zeta_{12}$) which is a $3721$-Weil number (resp.\ $5329$-Weil number). Then we can proceed by a similar argument as in the proof of (4) with these choices. \\
\indent (7) Take $G=C_{16}.$ Let $k=\F_{4}.$ Then $\pi:=2 \cdot \zeta_{16}$ is a $4$-Weil number, and hence, by Theorem \ref{thm HondaTata}, there exists a simple abelian variety $X$ over $\F_{4}$ of dimension $r$ such that $\pi_X$ is conjugate to $\pi.$ Then we have that $\Q(\pi_X)=\Q(\zeta_{16}),$ and hence, $\Q(\pi_X)$ has no real embeddings. Since $2$ is totally ramified in $\Q(\zeta_{16}),$ we can see that all the local invariants of $\textrm{End}_k^0(X)$ are zero, and hence, it follows from Proposition \ref{local inv}-(b) that $\textrm{End}_k^0(X)=\Q(\pi_X)$ and it is a CM-field of degree $2r$ over $\Q.$ Hence, we get $r=4$. Now, since $\Z[\zeta_{16}]$ is a (unique) maximal $\Z$-order in $\Q(\zeta_{16}),$ there exists a simple abelian variety $X^{\prime}$ over $k$ such that $X^{\prime}$ is $k$-isogenous to $X$ and $\textrm{End}_k (X^{\prime}) = \Z[\zeta_{16}]$ by \cite[Theorem 3.13]{17}. Note also that $C_{16} \cong \langle \zeta_{16} \rangle \leq \Z[\zeta_{16}]^{\times} = \Z^3 \times \langle \zeta_{16} \rangle = \textrm{Aut}_k (X^{\prime}).$ Then we can proceed as in the proof of (1) to see that $G$ is maximal in the isogeny class of $X.$ \\
\indent (8)-(10): For $G=C_{20}$ (resp.\ $C_{24}$), let $k=\F_4$ and let $\pi = 2 \cdot \zeta_{20}$ (resp.\ $\pi=2 \cdot \zeta_{24}$). For $G=C_{30},$ let $k=\F_{25}$ and let $\pi = 5 \cdot \zeta_{30}.$ Then we can proceed by a similar argument as in the proof of (7) with these choices. \\ 
\indent (11) Take $G=C_5 \rtimes C_{8}.$ Let $k=\F_{25}$ and let $\pi$ be a zero of the quartic polynomial $t^4 -30t^2 + 625 \in \Z[t]$ so that $\pi$ is a $25$-Weil number. By Theorem \ref{thm HondaTata}, there exists a simple abelian variety $X$ over $\F_{25}$ of dimension $r$ such that $\pi_X$ is conjugate to $\pi.$ Then we have that $K:=\Q(\pi_X)=\Q(2\sqrt{5}-\sqrt{-5})=\Q(\zeta_{20}+\zeta_{20}^9),$ and hence, $K$ has no real embeddings. Note also that $5$ splits into a product of two primes in $K$ and we can compute that $\textrm{inv}_{v}(\textrm{End}_k^0(X))=1/2$ for any place $\nu$ of $K$ lying over $5$. Hence it follows from Proposition \ref{local inv}-(b) that $D:=\textrm{End}_k^0(X)$ is a quaternion algebra over $\Q(\pi_X)$ (that is ramified exactly at the two places of $K$ lying over $5$), which, in turn, implies that $r=4$ by Corollary \ref{cor TateEnd0}-(b). Also, in view of \cite[Table 5]{Inneke}, by considering the ramification behavior of $D$ and the local Schur index, we can conclude that $G \leq D^{\times}.$ Furthermore, since $5$ splits into a product of two primes in $\Q(\zeta_{20})$, it follows from the first theorem of \cite[p. 407]{10} that $\Q(\zeta_{20})$ embeds into $D.$ \\
\indent  Now, let $\mathcal{O}$ be a maximal $\Z$-order in $D$ that contains the group $G.$ Then since $\mathcal{O}$ is a maximal $\Z$-order in $D,$ there exists a simple abelian variety $X^{\prime}$ over $k$ such that $X^{\prime}$ is $k$-isogenous to $X$ and $\textrm{End}_k(X^{\prime})=\mathcal{O}$ by \cite[Theorem 3.13]{17}. Note also that $G \leq \mathcal{O}^{\times}=\textrm{Aut}_k(X^{\prime})$ and $G$ is a maximal finite subgroup of $D^{\times}$ by Theorem \ref{deg2 case} and Lemma \ref{O,I not lem}. Then we can proceed as in the proof of (1) to see that $G$ is maximal in the isogeny class of $X.$  \\
\indent (12) Take $G=C_3 \rtimes C_{16}.$ Let $k=\F_{81}$ and let $\pi$ be a zero of the quartic polynomial $t^4 -126t^2 + 6561 \in \Z[t]$ so that $\pi$ is a $81$-Weil number. By Theorem \ref{thm HondaTata}, there exists a simple abelian variety $X$ over $\F_{81}$ of dimension $r$ such that $\pi_X$ is conjugate to $\pi.$ Then we have that $K:=\Q(\pi_X)=\Q(6\sqrt{2}-3\sqrt{-1})=\Q(\zeta_{8}),$ and hence, $K$ has no real embeddings. Note also that $3$ splits into a product of two primes in $K$ and we can compute that $\textrm{inv}_{v}(\textrm{End}_k^0(X))=1/2$ for any place $\nu$ of $K$ lying over $3$. Hence it follows from Proposition \ref{local inv}-(b) that $D:=\textrm{End}_k^0(X)$ is a quaternion algebra over $\Q(\pi_X)$ (that is ramified exactly at the two places of $K$ lying over $3$), which, in turn, implies that $r=4$ by Corollary \ref{cor TateEnd0}-(b). Also, in view of \cite[Table 5]{Inneke}, by considering the ramification behavior of $D$ and the local Schur index, we can conclude that $G \leq D^{\times}.$ Furthermore, since $3$ splits into a product of two primes in both $\Q(\zeta_{16})$ and $\Q(\zeta_{24})$, it follows from the first theorem of \cite[p. 407]{10} that both $\Q(\zeta_{16})$ and $\Q(\zeta_{24})$ embed into $D.$ Then we can proceed as in the proof of (11) to see that $G$ is maximal in the isogeny class of $X.$ \\
\indent (13) Take $G=C_5 \rtimes C_{16}.$ Let $k=\F_{625}$ and let $\pi$ be a zero of the quadratic polynomial $t^2 -30t+625 \in \Z[t]$ so that $\pi$ is a $625$-Weil number. Then by Theorem \ref{thm HondaTata}, there exists a simple abelian variety $X$ of dimension $r$ over $\F_{625}$ such that $\pi_X$ is conjugate to $\pi,$ and, in fact, we can have that $r=4$ by \cite[Proposition 2.5]{11}. Hence it follows that we have $\Q(\pi_X)=\Q(\sqrt{-1})$ and $D:=\textrm{End}_k^0(X)$ is a central simple division algebra of degree $4$ over $\Q(\sqrt{-1})$ by Lemma \ref{poss end alg}. (In fact, we can compute the local invariants directly to see that they are $1/4, 3/4$ at the two places of $K$ above $5$, and $0$ at other places of $K$.) Also, in view of \cite[Table 5]{Inneke}, by considering the ramification behavior of $D$ and the local Schur index, we can conclude that $G \leq D^{\times}.$ Furthermore, since $5$ splits completely in $\Q(\sqrt{-1})$ and splits into a product of two primes in both $\Q(\zeta_{16})$ and $\Q(\zeta_{20})$, it follows from the first theorem of \cite[p. 407]{10} that both $\Q(\zeta_{16})$ and $\Q(\zeta_{20})$ embed into $D.$ Then we can proceed as in the proof of (11) to see that $G$ is maximal in the isogeny class of $X.$\\
\indent This completes the proof.
\end{proof}

\begin{remark}
According to the proof of Theorem \ref{main thm}, we can see that all the candidates for the possible types of endomorphism algebras given in Lemma \ref{poss end alg} are actually realizable. As a consequence, the set
\begin{equation*}
\left\{m \in \Q~|~ m=\frac{2 \cdot \dim X}{[\textrm{End}_k^0(X):\Q]}~\textrm{for some simple abelian variety $X$ over a field $k$}\right\}
\end{equation*}
contains $\{1, 1/2, 1/4\}$ as a subset. In other words, this can be regarded as a concrete example of \cite[$\S$2]{Oort}. 
\end{remark}

\section{Application to Jordan constants}\label{Jordan sec}
In this section, we recall some of the facts in the theory of Jordan groups, following \cite{Popov}, and obtain one result on the values of Jordan constants of the automorphism groups of abelian fourfolds over finite fields, which can be regarded as an application of our main theorem. First, we recall the notions of Jordan groups and Jordan constants, which are of our main interests in this section.

\begin{definition}[{\cite[Definition 1]{Popov}}]\label{Jordan}
A group $G$ is called a \emph{Jordan group} if there exists an integer $d>0$, depending only on $G$, such that every finite subgroup $H$ of $G$ contains a normal abelian subgroup whose index in $H$ is at most $d.$ The minimal such $d$ is called the \emph{Jordan constant of $G$} and is denoted by $J_{G}$.
\end{definition}
In particular, it is easy to see that if every finite subgroup of $G$ is abelian, then $G$ is a Jordan group and $J_G =1.$ The following example illustrates this situation.

\begin{example}\label{ab exam}
Let $k$ be an algebraically closed field of characteristic zero, and let $X \subseteq \A_k^4$ be the nonsingular hypersurface defined by the equation $x_1^2 x_2 +x_3^2 + x_4^3 +x_1 =0.$ Then in view of \cite[\S2.2.5]{Popov}, every finite subgroup of $\textrm{Aut}(X)$ is cyclic, and hence, we have $J_{\textrm{Aut}(X)}=1.$
\end{example}
For our later use for computing the Jordan constants of certain infinite groups, we also record one useful result.
\begin{lemma}\label{Jor lem}
Let $G$ be a Jordan group. Then every subgroup $H$ of $G$ is also a Jordan group and we have
\begin{equation*}
J_G = \sup_{H \leq G} J_H
\end{equation*}  
where the supremum is taken over all finite subgroups $H$ of $G.$
\end{lemma}
\begin{proof}
The assertion that every subgroup of $G$ is also a Jordan group follows from \cite[Theorem 3]{Popov}. For the other assertion, let $\displaystyle d = \sup_{H \leq G} J_H,$ for convenience. First, we note that if $H$ is a finite subgroup of $G$, then we have $J_H \leq J_G.$ Hence it clearly follows that $d \leq J_G.$ To prove the reverse inequality, suppose on the contrary that $d< J_G.$ Then by the minimality of $J_G,$ we can see that there is a finite subgroup $H^{\prime}$ of $G$ such that $H^{\prime}$ contains no abelian normal subgroups of index $\leq d.$ Put $d^{\prime}:= J_{H^{\prime}}.$ Then since $d^{\prime} \leq d,$ it follows from the definition of $J_{H^{\prime}}$ that every finite subgroup of $H^{\prime}$ contains an abelian normal subgroup of index $\leq d^{\prime} \leq d.$ In particular, $H^{\prime}$ contains an abelian normal subgroup of index $\leq d$, which is a contradiction. Hence we get $J_G \leq d.$

\indent This completes the proof.
\end{proof}
This lemma is especially helpful if we have a great deal of information on the set of all finite subgroups of an infinite group $G.$ \\

The next lemma says that the multiplicative subgroup of certain central simple division algebras is a Jordan group.

\begin{lemma}\label{Jor lem2}
Let $X$ be a simple abelian variety of dimension $4$ over a finite field $k$. Let $D=\textrm{End}_k^0(X)$ and $K=\mathbb{Q}(\pi_X)$. Then $D^{\times}$ is a Jordan group.
\end{lemma}
\begin{proof}
In view of Lemma \ref{poss end alg}, we need to consider the following three cases: \\
(1) If $D$ is a CM-field of degree $8,$ then we have $D^{\times} \leq \mathbb{C}^{\times} = \textrm{GL}_1(\mathbb{C}).$ \\
(2) If $D$ is a central simple division algebra of degree $2$ over $K$, which is a quartic CM-field, then let $L$ be a maximal subfield of $D.$ By \cite[Theorem 16 in \S29]{Lor(2008)}, $L$ splits $D$ i.e.\ we have $D \otimes_K L \cong M_2(L).$ Then it follows that $D \otimes_K \mathbb{C} \cong M_2 (\mathbb{C})$ so that $D^{\times} \leq \textrm{GL}_2(\C).$ \\
Similarly, \\
(3) If $D$ is a central simple division algebra of degree $4$ over $K$, which is an imaginary quadratic field, then let $L$ be a maximal subfield of $D.$ Again, by \cite[Theorem 16 in \S29]{Lor(2008)}, we have $D \otimes_K L \cong M_4(L).$ Then it follows that $D \otimes_K \mathbb{C} \cong M_4 (\mathbb{C})$ so that $D^{\times} \leq \textrm{GL}_4(\C).$ \\
\indent Now, the desired result follows from \cite[Theorems 1 and 3]{Popov}. \\
\indent This completes the proof.
\end{proof}

 Now, we are ready to introduce our main result of this section.

\begin{theorem}\label{Jor thm}
Let $X$ be a simple abelian variety of dimension $4$ over a finite field $k$. Let $D=\textrm{End}_k^0(X)$ and $G=\textrm{Aut}_k(X).$ Then $G$ is a Jordan group and the Jordan constant $J_G$ of $G$ is contained in the set $\{1,2,4\}.$ 
\end{theorem}
\begin{proof}
Since $G \leq D^{\times},$ the first assertion follows from \cite[Theorem 3]{Popov} and Lemma \ref{Jor lem2}. For the second assertion, let $H$ be a maximal finite subgroup of $G.$ Then by Theorem \ref{main thm}, we know that $H$ must be one of the $13$ groups in Table 2 above. Also, in view of Lemma \ref{Jor lem}, it suffices to compute the Jordan constant $J_H$ of $H$. We consider the following three cases: \\
(1) If $H$ is cyclic i.e.\ if $H=C_n$ for $n\in \{2,4,6,8,10,12,16,20,24,30\},$ then we get $J_H = 1$. \\
(2) If $H=C_5 \rtimes C_8$ (resp.\ $H=C_3 \rtimes C_{16}$), then since all the normal (abelian) subgroups of $H$ are $C_n$ for $n \in \{1,2,4,5,10,20\}$ (resp.\ $C_n$ for $n \in \{1,2,3,4,6,8,12,24\}$) (see \cite{GN(1)}) (resp.\ see \cite{GN(2)}), it follows that we have $J_H \geq [H:C_{20}]=2$ (resp.\ $J_H \geq [H:C_{24}]=2).$ Now, if $J_H > 2$, in both cases, then by the minimality of $J_H$, it follows that there is a subgroup $H^{\prime}$ of $H$ such that $H^{\prime}$ contains no abelian normal subgroups of index $\leq 2.$ But since every subgroup of $H$ (in both cases) is cyclic, this is a contradiction. Thus we get $J_H =2.$ \\
(3) If $H=C_5 \rtimes C_{16},$ then since all the normal subgroups of $H$ are $C_5 \rtimes C_8$ and $C_n$ for $n \in \{1,2,4,5,10,20\}$, among which the abelian ones are just the cyclic groups (see \cite{GN(3)}), it follows that we have $J_H \geq [H:C_{20}]= 4.$ Then we can proceed as in the proof of (2) to get that $J_H = 4.$ \\
\indent This completes the proof.
\end{proof}
In fact, according to the proof of Theorem \ref{main thm}, the converse of Theorem \ref{Jor thm} is also true in the sense that if $n \in \{1,2,4\},$ then there is a simple abelian fourfold $X$ over a finite field $k$ such that $J_{\textrm{Aut}_k(X)} = n.$
\begin{remark}
In the spirit of \cite[Corollary 4.3]{Hwa(2020)}, it might be interesting to consider the case when the base field $k$ of a simple abelian fourfold $X$ is algebraically closed of positive characteristic (or characteristic zero).
\end{remark}

We conclude this paper with a group theoretic observation. To this aim, we recall that a finite group $H$ is called a \emph{Z-group} (resp.\ an \emph{A-group}) if all the Sylow subgroups of $H$ are cyclic (resp.\ abelian). In fact, it turns out that all the finite groups $H$ in Table 2 are Z-groups.
\begin{remark}
For an infinite group $G$ and an integer $n \geq 1$, we consider the following four properties: \\
($Z$) Every finite subgroup of $G$ is a Z-group. \\
($A$) Every finite subgroup of $G$ is an A-group. \\
($Z_n$) If $H$ is a finite subgroup of $G,$ then $H$ is a Z-group with $|H| \leq n$. \\
($A_n$) If $H$ is a finite subgroup of $G,$ then $H$ is an A-group with $|H| \leq n$. \\
\indent For an example of an infinite group $G$ with the property ($Z$), whence ($A$), we can take $G=\Z_{p^{\infty}}$, \emph{the Pr$\ddot u$fer $p$-group} for a prime $p.$ \\
\indent Now, let $G$ be an infinite group with the property ($Z_n$) for an integer $1 \leq n \leq 100$ (if it exists). Then we can show by Lemma \ref{Jor lem}, together with the knowledge of all Z-groups of order $\leq 100$ (see \cite{GN(4)}) that:\\
(i) If $1 \leq n \leq 41,$ then we have $J_G \leq 4$. \\
(ii) If $42 \leq n \leq 100,$ then we have $J_G \leq 6.$  \\
With this observation in mind, we can consider the following problems.
\begin{problem}\label{gp prob}
(a) Let $S$ be the set of integers $n \geq 1$ such that there is an infinite group $G$ with the property ($Z_n$). Is it true that $S \ne \phi?$ If so, can we describe the set $S$? In particular, is $S$ finite or infinite?\\
(b) Similarly, let $T$ be the set of integers $n \geq 1$ such that there is an infinite group $G$ with the property ($A_n$). Is it true that $T \ne \phi?$ If so, can we describe the set $T$? In particular, is $T$ finite or infinite? (Note that $S \subseteq T.$ Is it true that the inclusion is proper?) \\
(c) Suppose that $S$ is infinite. For each $n \in S$, let $\displaystyle J_n = \sup_{G} J_G$ where the supremum is taken over all infinite groups $G$ with the property ($Z_n$). Does there exist an integer $M>0$ such that $J_n \leq M$ as $n \rightarrow \infty$ with $n \in S$? If so, what is such an $M$? \\
(d) Let us define the following two sets 
$$U = \{(n,m) \in \Z^2~|~\textrm{there is an infinite group $G$ with the property $(Z_n)$ and $J_G = m$} \}$$ and 
$$V = \{(n,m) \in \Z^2~|~\textrm{there is an infinite group $G$ with the property $(A_n)$ and $J_G = m$}\}.$$
What can we say about those two sets $U$ and $V$? 
\end{problem} 
As an illustration, to see whether $(42,6) \in U$ or not, it seems that we need to determine whether we can produce an infinite group $G$ with the property $(Z_{42})$ containing the group $\textrm{F}_7 \cong \textrm{AGL}_1(\F_7)$, \emph{the Frobenius group of order $42$}. 
\end{remark}

\section*{Acknowledgements}
The author was supported by a KIAS individual Grant (MG069901) at Korea Institute for Advanced Study. He sincerely thanks Professor Jing Yu for asking a question on the subject during a conference at KAIST, which motivated this work.

\end{document}